\documentclass[11pt,a4paper]{amsart}

\usepackage{amsmath,amsthm,amssymb,mathrsfs,bbm}
\usepackage{lmodern, enumitem,multicol, multirow,setspace}
\usepackage{graphicx,tikz-cd}
\allowdisplaybreaks

\newtheorem{thm}{Theorem}[section]
\newtheorem{prop}[thm]{Proposition}
\newtheorem{cor}[thm]{Corollary}

\newtheorem{qst}[thm]{Question}

\theoremstyle{definition}

\newtheorem*{defn*}{Definition}

\DeclareMathOperator{\ran}{ran}

\DeclareMathOperator{\crit}{crit}
\DeclareMathOperator{\cf}{cf}

\DeclareMathOperator{\ot}{ot}

\DeclareMathOperator{\scl}{SCl}

\newcommand{\p}{\mathbb{P}}

\newcommand{\dotq}{\dot{\mathbb{Q}}}
\newcommand{\dotp}{\dot{\mathbb{P}}}

\newcommand{\lessd}{{<}\delta}

\DeclareMathOperator{\cl}{Cl}

\usepackage{hyperref}

\title{Woodin cardinals and forcing}

\author{Stamatis Dimopoulos}
\address{School of Mathematics, University of Bristol, University Walk, Bristol, BS8 1TW, UK}
\email{stamatis.dimopoulos@bristol.ac.uk}
\thanks{The author wishes to thank his supervisor Andrew Brooke-Taylor for his guidance and support while producing the results presented here.}

\begin{document}

\begin{abstract}
Despite being an established notion in the large cardinal hierarchy, results about Woodin cardinals are sparse in the literature. Here we gather known results about the preservation of Woodin cardinals under certain forcing extensions, as well as giving a template for preserving Woodin cardinals through forcing. Using this template, we form an indestructibility result under certain Easton iterations.
\end{abstract}

\maketitle

\section{Introduction} It is customary in set theory to establish results on the preservation of large cardinals in forcing extensions. The main use of such theorems is to check whether interesting combinatorial properties are compatible with the existence of large cardinals.

In this article, we focus at Woodin cardinals and their preservation in forcing extensions. A cardinal $\delta$ is Woodin if for every function $f:\delta\to\delta$, there is $\kappa<\delta$ such that $f``\kappa\subseteq \kappa$ and there is an elementary embedding $j:V\to M$ with critical point $\kappa$, such that $V_{j(f)(\kappa)}\subseteq M$. There is a naturally defined normal filter $F$ on $\delta$, where $X\subseteq \kappa$ is a member of $F$ iff the $\kappa$ in the above definition of Woodinness can be found in $X$. In fact, $\delta$ is Woodin iff $F$ is a $\delta$-complete normal filter on $\delta$ (see 26.15 in \cite{kanamori}).

Firstly, we will prove the following theorem, which gathers known results. The first clause appears in \cite{hamkins-woodin-strong} and the last clause generalises a result which appears in \cite{cody-easton-woodin}.

\begin{thm}\label{thm:woodin-forcing}
Suppose $\delta$ is a Woodin cardinal and $\p$ is a forcing notion that satisfies one of the following properties:
\begin{itemize}
	\item $\p$ has size less than $\delta$,
	\item $\p$ does not change $V_{\delta+1}$,
	\item $\p$ is $\delta$-strategically closed.
\end{itemize}
Then $\delta$ remains Woodin after forcing with $\p$.
\end{thm}

Moreover, we will describe a template for preserving Woodin cardinals through Easton iterations. We make substantial use of the ${}^\delta \delta$-bounding property to simplify our arguments and the point is that in extensions of such forcings, it suffices to verify the definition of Woodinness for functions in the ground model. Using this template we can simplify proofs, such as that of the following theorem which appeared in \cite{cody-easton-woodin}.

\begin{thm}[Cody, \cite{cody-easton-woodin}]\label{thm:woodin-easton}
Assume GCH holds, $\delta$ be a Woodin cardinal and $F:Reg\to Card$ is an Easton function such that $F``\delta\subseteq \delta$. Then there is a generic extension in which $\delta$ is still Woodin and for every regular cardinal $\gamma$, $2^\gamma=F(\gamma)$.
\end{thm}

Moreover, we show the following indestructibility result. We need the notion of a \emph{strong closure point} of a forcing iteration $\p$, which is an ordinal $\alpha$ such that
\begin{enumerate}
	\item $\p_\alpha\subseteq V_\alpha$,
	\item for all $\beta\geq\alpha$, $\Vdash_\beta \dotq_\beta$ is $<\alpha^+$-strategically closed.
\end{enumerate} 

\begin{thm}\label{thm:woodin-indesructibility}
Suppose $\delta$ is a Woodin cardinal, GCH holds and $\p$ is an Easton support $\delta$-iteration which satisfies:
\begin{enumerate}
	\item $\p\subseteq V_\delta$
	\item the set of strong closure points of $\p$ is a member of the Woodin filter of $\delta$.
\end{enumerate}
Then $\delta$ remains Woodin after forcing with $\p$.
\end{thm}

With the usual clo

The notation we use is fairly standard. We will occasionally use interval notation $(\alpha,\beta)$ for two ordinals $\alpha<\beta$, to denote the set $\{\xi\mid \alpha<\xi<\beta\}$. For the properties of $\alpha$-strategically closed, ${<}\alpha$-strategically closed and $\alpha$-distributive forcings, we address the reader to Section 5 in \cite{cummings-handbook}. 
A forcing $\p$ is called ${}^\delta \delta$-bounding if every function $f:\delta\to\delta$ in the forcing extension by $\p$ is bounded by some ground model function, i.e.
$$\Vdash_\p \forall \dot{f}\in {}^\delta \delta~ \exists g\in {}^\delta \delta\cap V ~\forall \alpha<\delta (\dot{f}(\alpha)<g(\alpha)).$$
It is easy to see that if $\p$ is $\delta$-c.c. it is ${}^\delta \delta$-bounding.

The large cardinal notions we deal with are witnessed by the existence of elementary embeddings of the form $j:V\to M$, where $V$ is the universe we work in and $M\subseteq V$ is a transitive class. The critical point of an elementary embedding $j$ is denoted by $\crit(j)$.

For two cardinals $\kappa,\lambda$ we say $\kappa$ is \emph{$\lambda$-strong} if there is $j:V\to M$ with $\crit(j)=\kappa$, $j(\kappa)>\lambda$ and $V_\lambda\subseteq M$. We will also say $\kappa$ is \emph{${<}\mu$-strong} if it is $\lambda$-strong for all $\lambda<\mu$. We will always assume that $\lambda\geq\kappa$ even when not mentioned explicitly. 

A cardinal $\kappa$ is called \emph{$\lambda$-strong for $A$}, where $A$ is any set, if there a $\lambda$-strongness embedding $j:V\to M$ with $\crit(j)=\kappa$, satisfying the property $A\cap V_\lambda=j(A)\cap V_\lambda$. Once again, we use expressions like $\kappa$ is ${<}\mu$-strong for $A$ to mean that $\kappa$ is $\lambda$-strong for $A$, for all $\lambda<\mu$ and it is always assumed that $\lambda\geq\kappa$. We are going to make use of the following results.

\begin{prop}\label{prop:a-strong-extender}
Suppose $j:V\to M$ is a $\lambda$-strongness for $A$ embedding with $\crit(j)=\kappa$, where $\lambda\geq\kappa$ is a $\beth$-fixed point. If $E$ is the $(\kappa,\lambda)$-extender derived from $j$, then the extender embedding $j_E:V\to M_E$ is also $\lambda$-strong for $A$. Moreover, if $\cf(\lambda)>\kappa$, then ${}^\kappa M_E\subseteq M_E$.
\end{prop}

\begin{proof}
See Remark 1 in \cite{cody-easton-woodin}.
\end{proof}

Since our proofs involve lifting elementary embeddings though forcing extensions, we give the following two results about the construction of generic filters.

\begin{prop}[Diagonalisation]\label{thm:diagonalisation}
Suppose $M\subseteq V$ is an inner model, $\p\in M$ is a forcing notion and $p\in \p$. If
\begin{enumerate}
\item ${}^\kappa M\subseteq M$
\item $\p$ is ${<}\kappa^+$-strategically closed in $M$
\item there are at most $\kappa^+$-many maximal antichains of $\p$ in $M$, counted in $V$,
\end{enumerate}
then there is in $V$, an $M$-generic filter $H\subseteq \p$ such that $p\in H$.
\end{prop}

\begin{proof}
See Proposition 8.1 in \cite{cummings-handbook} or Theorem 51 in \cite{hamkins-largecardinals}.
\end{proof}

\begin{prop}\label{thm:transering-generic}
Suppose $j:V\to M$ is an elementary embedding with $\crit(j)=\kappa$,  generated by a $(\kappa,\lambda)$-extender. If $\p$ is a $\kappa^+$-distributive forcing notion and $G\subseteq \p$ is a $V$-generic filter, then the filter generated by $j``G$ is $M$-generic for $j(\p)$.
\end{prop}

\begin{proof}
See Proposition 15.1 in \cite{cummings-handbook}.
\end{proof}

We finish the section by mentioning the fact that the property that $\delta$ is a Woodin-like cardinal can be verified in $V_{\delta+1}$. This means that if for any $M\subseteq V$ with $V_{\delta+1}\subseteq M$, $M\models ``\delta$ is Woodin".

\section{Proof of Theorem \ref{thm:woodin-forcing}}

We only need to show the last clause, since the first can be found in \cite{hamkins-woodin-strong} and the second, follows from the fact that the Woodinness of $\delta$ can be verified in $V_{\delta+1}$. 

Assume $\p$ is a $\delta$-strategically closed forcing. Let $G\subseteq \p$ be a $V$-generic filter and let $f:\delta\to \delta$ be a function in $V[G]$. We claim that for each $p\in \p$, there is a decreasing sequence $\langle p_\alpha\mid \alpha<\delta\rangle$ of conditions extending $p$, such that $p_\alpha$ decides $\dot{f}\restriction (\alpha+1)$.

To see this, consider the following play $\langle p_\alpha\mid \alpha<\delta\rangle$ of the game $G_\delta(\p)$. At stage $1$, ODD plays a condition $p_1$ that extends $p$ and decides $\dot{f}(0)$. At limit or even successor stages, by the $\delta$-strategic closure of $\p$, EVEN can play a condition according to his winning strategy. Let $\langle\beta_\xi:\xi<\delta\rangle$ be an increasing enumeration of the odd ordinals in $\delta$. If it is ODD's turn at stage $\beta_\xi$, ODD plays a condition $p_{\beta_\xi}$ extending the previous move of EVEN and such that $p_{\beta_\xi}$ decides $\dot{f}(\xi)$. Since $p_{\beta_\xi}$ will extend all the previous conditions, it actually decides $\dot{f}\restriction (\xi+1)$. Thus, $\langle p_\alpha\mid \alpha<\delta\rangle$ is the required sequence.

Now fix $p\in G$ and let $\langle p_\alpha\mid \alpha<\delta\rangle$ be the sequence generated as above. Let $g:\delta\to\delta$ be the function in $V$ determined by the $p_\alpha$'s and using the Woodinness of $\delta$, let $\kappa<\delta$ be a closure point of $g$ and $j:V\to M$ an elementary embedding with $\crit(j)=\kappa$ and $V_{j(g)(\kappa)}\subseteq M$. As usual, we can assume that $j$ is an extender embedding. Since $\p$ is $\delta$-strategically closed it is $\delta$-distributive and in particular $\kappa^+$-distributive, so using \ref{thm:transering-generic} we can lift $j$ to $j:V[G]\to M[j(G)]$. Since $j(g)(\kappa)<\delta$ the $\delta$-distributivity implies that $(V_{j(g)(\kappa)})^{V[G]}=V_{j(g)(\kappa)}$.  

Now note that by the choice of $p$ and the definition of the $p_\alpha$'s, 
$$f\restriction \kappa=(\dot{f})_G \restriction \kappa=g\restriction \kappa.$$ Thus $\kappa$ is also a closure point of $f$ in $V[G]$ and by elementarity $j(f)(\kappa)=j(g)(\kappa)$. Therefore, $j:V[G]\to M[j(G)]$ is an embedding with the property $(V_{j(f)(\kappa)})^{V[G]}\subseteq M[j(G)]$. As $f$ was chosen arbitrarily, this shows that $\delta$ is Woodin in $V[G]$.

\section{Description of the template}

\begin{prop}[Template for preserving Woodinness]\label{prop:template}
Suppose $\delta$ is Woodin, $\p$ is a ${}^\delta \delta$-bounding forcing and $G\subseteq \p$ is a $V$-generic filter. Then, $\delta$ is Woodin in $V[G]$ if the following property holds in $V[G]$: 
\begin{center}
for every $f:\delta\to\delta$, $f\in V$, there are $\kappa\leq\lambda<\delta$ such that $f``\lambda\subseteq \lambda$ and a $\lambda$-strongness for $g$ embedding $j:V[G]\to M$ with $\crit(j)=\kappa$ and $\lambda> f(\kappa)$.
\end{center}
\end{prop}

\begin{proof}
Let $h:\delta\to\delta$ be a function in $V[G]$ and let $f:\delta\to\delta$ be a function in $V$ that bounds $h$. By our assumption, there is a $\lambda$-strongness for $g$ embedding $j:V[G]\to M$ with $\crit(j)=\kappa$, for some $\lambda>f(\kappa)$ such that $f``\lambda\subseteq \lambda$. 

Since $j(f)\cap V_\lambda=f\cap V_\lambda$, $\lambda$ is also a closure point of $j(f)$ and for every $\alpha<\kappa$,
$$j(f)(j(\alpha))=j(f)(\alpha)<\lambda<j(\kappa).$$
By elementarity, $f(\alpha)<\kappa$ and so, $\kappa$ is a closure point of $f$. Since $f$ bounds $h$, $\kappa$ is a closure point of $h$ too. Finally, $j(h)(\kappa)<j(g)(\kappa)<\lambda$ and so, $(V_{j(h)(\kappa)})^{V[G]}\subseteq M$. As $f$ was chosen arbitrarily, $\delta$ is Woodin in $V[G]$.
\end{proof}

Thus, to verify the Woodinness of $\delta$ in ${}^\delta \delta$-bounding extensions, we only need to look at functions that exist in the ground model and find an appropriate strongness embedding for each of them. 

In our applications, in order to find these embeddings, we use \ref{prop:defns-woodin} below to fix such embeddings in $V$ and lift them through $\p$. 


\begin{prop}\label{prop:defns-woodin}
The following are equivalent for a cardinal $\delta$.
\begin{enumerate}
\item $\delta$ is Woodin, i.e. for every function $f:\delta\to \delta$ there is $\kappa<\delta$ which is a closure point of $f$ and there is an elementary embedding $j:V\to M$ with $\crit(j)=\kappa$ and $V_{j(f)(\kappa)}\subseteq M$.
\item For every $A\subseteq V_\delta$, there is $\kappa<\delta$ which is $\lessd$-strong for $A$.
\end{enumerate}
\end{prop}

\begin{proof}
See 26.14 in \cite{kanamori}.
\end{proof}

As a first application, let us see how to force GCH while preserving a Woodin cardinal $\delta$. By the second clause of \ref{thm:woodin-forcing}, it suffices to show that the Woodinness of $\delta$ is preserved after we force GCH below $\delta$.

Let $\langle \p_\alpha,\dotq_\beta\mid \alpha\leq\delta,\beta<\delta\rangle$ be the Easton support $\delta$-iteration, where $\dotq_\alpha$ is a name for $Add(\alpha^+,1)$ whenever $\Vdash_{\p_\alpha} ``\alpha$ is a cardinal", and trivial otherwise.

Let $G\subseteq \p$ be a $V$-generic filter. Since $\delta$ is Mahlo and we are using Easton support, it follows that $\p$ is $\delta$-c.c. and so, ${}^\delta \delta$-bounding. We will show that $\delta$ remains Woodin in $V[G]$ by using the template, so let $f:\delta\to\delta$ be a function in $V$ and let $j:V\to M$ be a $\lambda$-strongness for $f$ embedding with $\crit(j)=\kappa$, where $\lambda\in (\kappa,\delta)$ is an inaccessible cardinal such that $f``\lambda\subseteq \lambda$. By \ref{prop:a-strong-extender}, we can assume that $j$ is an extender embedding by a $(\kappa,\lambda)$-extender and that ${}^\kappa M\subseteq M$.

We aim to lift $j$ through $\p$ and for this end, we factorise $\p_\delta$ as $\p_\kappa\ast \dotp_{\geq\kappa}$, where $\dotp_{\geq\kappa}$ is a name for the stages greater than or equal to $\kappa$. 

To lift $j$ through $\p_\kappa$, note that by elementarity, $j(\p_\kappa)$ is the GCH forcing up to $j(\kappa)$. Since $V_\lambda\subseteq M$ and $\lambda$ is inaccessible in both $V$ and $M$, the first $\lambda$-stages of $j(\p_\kappa)$ are the same as for $\p$. Thus, we can write $j(\p_\kappa)$ as $\p_\lambda\ast \dotp_{tail}$, where $\dotp_{tail}$ is naming the stages greater than or equal to $\lambda$. Using $G_\lambda$ as an $M$-generic filter, we can form $M[G_\lambda]$. Now, we claim that in $V[G_\lambda]$ there is an $M[G_\lambda]$-generic filter $H_1$ for $\p_{tail}=(\dotp_{tail})_{G_\lambda}$. 

Write $\p_\lambda$ as $\p_\kappa\ast \dotp_{[\kappa,\lambda)}$. We know that $V\models {}^\kappa M\subseteq M$ and since $\p_\kappa$ has the $\kappa$-c.c. and $\p_{[\kappa,\lambda)}$ is $\kappa^+$-distributive in $V[G_\kappa]$, the usual arguments show that $V[G_\lambda]\models {}^\kappa M[G_\lambda]\subseteq M[G_\lambda]$. Moreover, $\p_{tail}$ is (much more than) $\kappa^+$-strategically closed in $M[G_\lambda]$. An easy counting argument in $V[G_\lambda]$, shows that $\p_{tail}$ has at most $\kappa^+$-many maximal antichains in $M[G_\lambda]$. Thus, by \ref{thm:diagonalisation} it is possible to construct a generic filter $H_1\subseteq \p_{tail}$. Since $j``G_\kappa=G_\kappa\subseteq G_\lambda\ast H_1$, we can lift $j$ to $j:V[G_\kappa]\to M[j(G)]$, where $j(G)=G_\lambda\ast H_1$.

To further lift $j$, note that $\p_{\geq\kappa}=(\dotp_{\geq\kappa})_{G_\kappa}$ is $\kappa^+$-distributive and so, by \ref{thm:transering-generic}, $j``G_{\geq\kappa}$ generates an $M[j(G_\kappa)]$-generic filter $H_2$ for $j(\p_{\geq\kappa})$. This allows us to further lift $j$ to $j:V[G]\to M[j(G)]$, where $j(G)=G_\lambda\ast H_1\ast H_2$. Since $(V_\lambda)^{V[G]}=(V_\lambda)^{V[G_\lambda]}=V_\lambda[G_\lambda]\subseteq M[j(G)]$, it follows that $j$ is a $\lambda$-strongness embedding. Also $j(f)\cap V_\lambda=f\cap V_\lambda$, because $f\in V$ and $j\restriction V$ had the same property. Hence $j$ is a $\lambda$-strongness for $f$ embedding. 
Thus, we have verified the template \ref{prop:template} and so, $\delta$ is Woodin in $V[G]$.

\section{Proof of Cody's theorem \ref{thm:woodin-easton}.}


By \ref{thm:woodin-forcing}, it suffices to realise $F$ below $\delta$, as to realise it at ordinals greater than or equal $\delta$ we can use Easton's original forcing which is $\delta$-strategically closed.

We use the forcing notion $\p=\langle \p_\alpha,\dotq_\beta\mid \alpha\leq\delta,\beta<\delta\rangle$ found in \cite{cody-easton-woodin}, which is defined as follows. Suppose $\p_\alpha$ has been defined. $\dotq_\alpha$ is non-trivial only when $\alpha$ is a closure point of $F$. In that case, let $\mu$ be the first closure point of $F$ above $\alpha$. If $\alpha$ is an inaccessible closure point of $F$, let $\dotq_\alpha$ name the Easton product\footnote{For the exact definition of Sacks forcing, see \cite{cody-easton-woodin} or \cite{friedman-thompson-perfect-trees}.} 
\begin{equation*}
Sacks(\kappa,F(\kappa))\times \prod_{\substack{\alpha<\lambda<\mu \\ \lambda\text{ regular}}} Add(\lambda, F(\lambda)).
\end{equation*}
If $\alpha$ is a singular closure point of $F$, let $\dotq_\alpha$ name the Easton product 
$$\prod_{\substack{\alpha<\lambda<\mu \\ \lambda \text{ regular}}} Add(\lambda,F(\lambda)).$$

Let $G\subseteq \p$ be a $V$-generic filter. We are going to show that $\delta$ remains Woodin by applying the template \ref{prop:template}. So fix $h:\delta\to\delta$ in $V$ and let $\kappa<\delta$ be a $\lessd$-strong for $F \sqcup h$ cardinal. Pick an inaccessible cardinal $\lambda\in (\kappa,\delta)$ such that $h``\lambda\subseteq \lambda$ and $F``\lambda\subseteq \lambda$ and let $j:V\to M$ be a $\lambda$-strongness for $F\sqcup h$ embedding with $\crit(j)=\kappa$.

To lift $j$ through $\p_\kappa$, note that by elementarity, $j(\p_\kappa)$ is realising $j(F)$ in $M$ below $j(\kappa)$. Since $j(F)\cap V_\lambda=F\cap V_\lambda$ and $\lambda$ is a closure point of $F$, it follows that $j(F)\restriction \lambda=F\restriction \lambda$. Also, $\lambda$ is inaccessible in both $V$ and $M$, so the first $\lambda$-stages of $j(\p_\kappa)$ are the same as for $\p$. Thus, we can write $j(\p_\kappa)$ as $\p_\lambda\ast \dotp_{tail}$, where $\dotp_{tail}$ is naming the stages greater than or equal to $\lambda$. Using $G_\lambda$ as an $M$-generic filter, we can form $M[G_\lambda]$. As in the proof of GCH above, we can construct in $V[G_\lambda]$ an $M[G_\lambda]$-generic filter $H_1$ for $\p_{tail}=(\dotp_{tail})_{G_\lambda}$. 
Since $j``G_\kappa=G_\kappa\subseteq G_\lambda\ast H_1$, we can lift $j$ to $j:V[G_\kappa]\to M[j(G_\kappa)]$, where $j(G_\kappa)=G_\lambda\ast H_1$. 

Let $g$ be the $V[G_\kappa]$-generic filter for $Sacks(\kappa,F(\kappa))$. To further lift $j$ through $Sasks(\kappa,F(\kappa))$, we take advantage of the specific properties of Sacks forcing and in particular the fusion property. We give here a sketch, see Section 3.2 in \cite{cody-easton-woodin} and Lemmata 4-6 in \cite{friedman-thompson-perfect-trees} for the details. 

For $\alpha<j(F(\kappa))$, if we let $t_\alpha=\bigcap_{p\in G_\kappa} j(p)(\alpha)$ then there are two cases:
\begin{enumerate}
	\item $\alpha\in j``F(\alpha)$, in which case $t_\alpha$ is a \emph{tuning fork}, i.e. a tree on $2^{<j(\kappa)}$ that is the union of two cofinal branches that split at $\kappa$,
	\item $\alpha\not\in \ran(j)$, in which case $t_\alpha$ is a cofinal branch in $2^{<j(\kappa)}$.
\end{enumerate}

In both cases, $t_\alpha$ has a branch $t'_\alpha\subseteq t$ that extends $j(p)(\alpha)$ for all $\alpha$. The point now is that the filter 
$$H_2=\{p\in j(Sacks(\kappa,F(\kappa)))\mid \forall \alpha<j(F(\kappa)) (t'(\alpha)\subseteq p(\alpha))\}$$
is $M[j(G_\kappa)]$-generic for $j(Sacks(\kappa,F(\kappa)))$ and contains $j``g$. Since this generic is constructed in $V[G_\lambda]$, we can lift $j$ to $j:V[G_\kappa][g]\to M[j(G_\kappa)][j(g)]$, where $j(g)=H_2$.

Now the rest of the forcing is $\kappa^+$-distributive, so by \ref{thm:transering-generic} we can use the filter $H_3$ generated by $j``G_{>\kappa}$ to lift $j$ to $j:V[G]\to M[j(G)]$, where $j(G)=G_\lambda\ast H_1\ast H_2\ast H_3$. Note that $(V_\lambda)^{V[G]}=V_\lambda[G_\lambda]\subseteq M[G_\lambda]\subseteq M[j(G)]$ and so, $j$ is $\lambda$-strongness embedding. Also $j(h)\cap V_\lambda=h\cap V_\lambda$, because $h\in V$ and $j\restriction V$ had the same property. As $h$ was chosen arbitrarily, we have verified the template \ref{prop:template} and so, $\delta$ remains Woodin in $V[G]$. \hfill $\qed$

\section{Proof of Theorem \ref{thm:woodin-indesructibility}}

Let $\p=\langle \p_\alpha,\dotq_\beta\mid \alpha\leq\delta,\beta<\delta\rangle$ be an Easton support $\delta$-iteration as in \ref{thm:woodin-indesructibility}. We are going to show that $\delta$ remains Woodin after forcing with $\p$ by using the template \ref{prop:template}. Let $G\subseteq \p$ be a $V$-generic filter and let $f:\delta\to\delta$ be a function in $V$. 

Let $\scl(\p)$ denote the set of strong closure points of $\p$. With the usual closure arguments, we can show that $\scl(\p)$ is a club in $\delta$ and so, it is a member of the Woodin filter of $\delta$. Hence, we can find $\kappa<\delta$ which is $\lessd$-strong for $f\sqcup \p\sqcup \scl(\p)$ and is itself a strong closure point of $\p$. Let $\lambda\in(\kappa,\delta)$ be an inaccessible cardinal, such that it is a limit point of $\cl(p)$ and $\lambda>f(\kappa)$. Let $j:V\to M$ be a $\lambda$-strongness for $f\sqcup \p\sqcup \scl(\p)$ embedding with $\crit(j)=\kappa$. By \ref{prop:a-strong-extender} we can assume $j$ is given by a $(\kappa,\lambda)$-extender and ${}^\kappa M\subseteq M$.

We factorise $\p$ as $\p_\kappa\ast \dotp_{\geq\kappa}$. We first lift $j$ through $\p_\kappa$. Note that $\lambda$ is inaccessible in both $V$ and $M$, so we use a direct limit at the $\lambda$-stage of both $\p$ and $j(\p)$. Also, $j(\p)\cap V_\lambda=\p\cap V_\lambda$ and $\p_\lambda\subseteq V_\lambda$, thus $j(\p_\kappa)\simeq \p_\lambda\ast \dotp_{tail}$, where $\dotp_{tail}$ is a name for the stages greater than or equal to $\lambda$.

Write $\p_\lambda$ as $\p_\kappa\ast \dotp_{[\kappa,\lambda)}$. Since $G_\lambda$ is $V$-generic for $\p_\lambda$ it is also $M$-generic, so we can form $M[G_\lambda]$. We will now construct an $M[G_\lambda]$-generic filter for $\p_{tail}=(\dotp_{tail})_{G_\lambda}$ by using \ref{thm:diagonalisation}. Firstly, note that ${}^\kappa M\subseteq M$ and $\p_\kappa$ is $\kappa$-c.c., so $V[G_\kappa]\models {}^\kappa M[G_\kappa]\subseteq M[G_\kappa]$. Also, $\p_{[\kappa,\lambda)}$ is $\kappa^+$-distributive (since $\kappa$ is a strong closure point of $\p$) and so, $V[G_\lambda]\models {}^\kappa M[G_\lambda]\subseteq M[G_\lambda]$. By the fact that $j(\scl(\p))\cap V_\lambda=\scl(\p)\cap V_\lambda$, we know that there are strong closure points of $j(\p)$ in $(\kappa,\lambda)$ an so, $\p_{tail}$ is at least ${<}\kappa^+$-strategically closed in $M[G_\lambda]$. Finally, an easy counting argument using GCH, shows that $\p_{tail}$ has at most $\kappa^+$-many maximal antichains in $M[G_\lambda]$, counted in $V[G_\lambda]$.\footnote{Here is the only part of the proof that uses the fact that GCH holds in $V$.} Thus, by \ref{thm:diagonalisation} we can construct an $M[G_\lambda]$-generic filter $H_1$ for $\p_{tail}$. Since $j``G_\kappa=G_\kappa\subseteq G_\lambda\ast H_1$, we can lift $j_1$ to $j_1:V[G_\kappa]\to M[j(G_\kappa)]$, where $j(G_\kappa)=G_\lambda\ast H_1$.

The rest of $\p$, i.e. $\p_{\geq\kappa}$ is $\kappa^+$-distributive (since $\kappa$ is a strong closure point of $\p$), so by \ref{thm:transering-generic}, $j``G_{\geq\kappa}$ generates an $M[j(G_\kappa)]$-generic filter $H_2$ for $j(\p_{\geq\kappa})$, which allows us to lift $j$ to $j:V[G]\to M[j(G)]$, where $j(G)=G_\lambda\ast H_1\ast H_2$.

Since $(V_\lambda)^{V[G]}=(V_\lambda)^{V[G_\lambda]}=V_\lambda[G_\lambda]\subseteq M[G_\lambda]\subseteq M[j(G)]$, $j$ is a $\lambda$-strongness embedding. Moreover, $j(f)\cap V_\lambda=f\cap V_\lambda$, because $f\in V$ and $j\restriction V$ had the same property. As $f$ was chosen arbitrarily, the template \ref{prop:template} holds and so, $\delta$ is Woodin in $V[G]$. \hfill $\qed$

\medskip

There are various forcing notions that satisfy the conditions of \ref{thm:woodin-indesructibility}, we give below a small sample.

\begin{cor}
If $\delta$ is Woodin, then it remains so after forcing the following:
\begin{enumerate}
	\item $\square_\kappa$ for all infinite cardinals $\kappa<\delta$,
	\item $\diamondsuit^+_{\kappa^+}$ for all infinite cardinals $\kappa<\delta$,
	\item every non-strong regular cardinal has a non-reflecting stationary subset,
	\item $V=HOD$.
\end{enumerate}
\end{cor}

\begin{proof}
All the above properties can be forced with Easton iterations whose iterands are sufficiently strategically closed. For adding $\square$-sequences, see 6.6 in \cite{cummings-handbook}, for adding $\diamondsuit^+$-sequences see Section 12 in \cite{cummings-foreman-magidor}, for adding non-reflecting stationary sets see Section 1 in \cite{identity-crises-I} and for forcing $V=HOD$ see \cite{brooke-taylor-definable-wellorder} or \cite{hamkins-wholeness}.
\end{proof}
 
\section{Questions}

We will conclude with a question based on the following theorem of Brooke-Taylor in \cite{brooke-taylor-indestructibility}. 

\begin{thm}\label{thm:bt}
Suppose $\delta$ is a Vop\v enka cardinal and $\p=\langle \p_\alpha,\dotq_\beta\mid \alpha\leq\delta,\beta<\delta\rangle$ is an Easton support $\delta$-iteration such that 
\begin{enumerate}
	\item for all $\alpha<\delta$, $\dotq_\alpha$ has size less than $\delta$
	\item for all $\alpha<\delta$ there is $\beta\geq\alpha$ such that $\forall \gamma\geq\beta$, $\dotq_\beta$ is naming an $\alpha$-directed closed forcing.
\end{enumerate}
Then $\delta$ remains Vop\v enka after forcing with $\p$.
\end{thm}

Since Vop\v enka cardinals are the Woodinised analogue for supercompactness (see \cite{perlmutter-2015}) it is natural to ask if the same iterations preserve Woodin cardinals.

\begin{qst}
If $\delta$ is Woodin and $\p$ is an iteration as in \ref{thm:bt}, does $\delta$ remain Woodin after forcing with $\p$?
\end{qst}

We anticipate that a negative answer, i.e. a counterexample, would have to be based on some substantial difference between the two notions, such as the existence of canonical inner models for the one and not the other.

\bibliographystyle{plain}
\bibliography{math_refs}

\begin{thebibliography}{10}

\bibitem{identity-crises-I}
Arthur~W. Apter and James Cummings.
\newblock Identity crises and strong compactness.
\newblock {\em J. Symbolic Logic}, 65(4):1895--1910, 2000.

\bibitem{brooke-taylor-definable-wellorder}
Andrew~D. Brooke-Taylor.
\newblock Large cardinals and definable well-orders on the universe.
\newblock {\em J. Symbolic Logic}, 74(2):641--654, 2009.

\bibitem{brooke-taylor-indestructibility}
Andrew~D. Brooke-Taylor.
\newblock Indestructibility of {V}op\v enka's {P}rinciple.
\newblock {\em Arch. Math. Logic}, 50(5-6):515--529, 2011.

\bibitem{cody-easton-woodin}
Brent Cody.
\newblock Easton's theorem in the presence of {W}oodin cardinals.
\newblock {\em Arch. Math. Logic}, 52(5-6):569--591, 2013.

\bibitem{cummings-handbook}
James Cummings.
\newblock Iterated forcing and elementary embeddings.
\newblock In {\em Handbook of set theory. {V}ols. 1, 2, 3}, pages 775--883.
  Springer, Dordrecht, 2010.

\bibitem{cummings-foreman-magidor}
James Cummings, Matthew Foreman, and Menachem Magidor.
\newblock Squares, scales and stationary reflection.
\newblock {\em J. Math. Log.}, 1(1):35--98, 2001.

\bibitem{friedman-thompson-perfect-trees}
Sy-David Friedman and Katherine Thompson.
\newblock Perfect trees and elementary embeddings.
\newblock {\em J. Symbolic Logic}, 73(3):906--918, 2008.

\bibitem{hamkins-largecardinals}
Joel~David Hamkins.
\newblock Forcing and large cardinals.
\newblock manuscript.

\bibitem{hamkins-wholeness}
Joel~David Hamkins.
\newblock The wholeness axioms and {$V=\rm HOD$}.
\newblock {\em Arch. Math. Logic}, 40(1):1--8, 2001.

\bibitem{hamkins-woodin-strong}
Joel~David Hamkins and W.~Hugh Woodin.
\newblock Small forcing creates neither strong nor {W}oodin cardinals.
\newblock {\em Proc. Amer. Math. Soc.}, 128(10):3025--3029, 2000.

\bibitem{kanamori}
Akihiro Kanamori.
\newblock {\em The higher infinite}.
\newblock Springer Monographs in Mathematics. Springer-Verlag, Berlin, second
  edition, 2009.
\newblock Large cardinals in set theory from their beginnings, Paperback
  reprint of the 2003 edition.

\bibitem{perlmutter-2015}
Norman~Lewis Perlmutter.
\newblock The large cardinals between supercompact and almost-huge.
\newblock {\em Arch. Math. Logic}, 54(3-4):257--289, 2015.

\end{thebibliography}

\end{document}